\documentclass[leqno,12pt]{article}

\usepackage{latexsym}
\usepackage{amssymb}
\usepackage{amsthm}
\usepackage{amsmath}
\usepackage{relsize}
\usepackage{color}
\newcommand{\nc}{\newcommand}
\nc{\nt}{\newtheorem}
\nt{thm}{Theorem}[section]
\nt{cor}[thm]{Corollary}
\nt{prop}[thm]{Proposition}
\nt{obs}[thm]{Observation}
\nt{claim}[thm]{Claim}
\nt{lem}[thm]{Lemma}
\nt{defn}[thm]{Definition}
\nt{exa}[thm]{Example}
\nt{rem}[thm]{Remark}
\nt{ass}[thm]{Assumption}
\nt{alg}[thm]{Algorithm}
\nt{con}[thm]{Conjecture}
\nc{\ip}[2]{\mbox{$\langle #1,#2 \rangle$}}
\nc{\pf}{\noindent{\bf Proof\ \ }}
\nc{\finpf}{\hfill{$\Box$}\linespace}
\nc{\linespace}{\vspace{\baselineskip} \noindent}
\nc{\R}{{\bf R}}
\nc{\cl}{\mbox{\rm cl}\,}
\nc{\cls}{ \mbox{{\scriptsize {\rm cl}}}\,}
\nc{\conv}{\mbox{\rm conv}}
\nc{\rb}{\mbox{\rm rb}\,}
\nc{\ri}{\mbox{\rm ri}\,}
\nc{\inter}{\mbox{\rm int}\,}
\nc{\kernel}{\mbox{\rm ker}\,}
\nc{\range}{\mbox{\rm range}\,}
\nc{\bd}{\mbox{\rm bd}\,}
\nc{\spann}{\mbox{\rm span}\,}
\nc{\rint}{\mbox{\rm rint}\,}
\nc{\epi}{\mbox{\rm epi}\,}
\nc{\gph}{\mbox{\rm gph}\,}
\nc{\rge}{\mbox{\rm rge}\,}
\nc{\rgel}{\mbox{\rm {\scriptsize rge}}\,}
\nc{\sepi}{\mbox{\rm {\scriptsize epi}}\,}
\nc{\sbd}{\mbox{\rm {\scriptsize bd}}\,}
\nc{\dom}{\mbox{\rm dom}\,}
\nc{\lin}{\mbox{\rm lin}\,}
\nc{\detr}{\mbox{\rm det}\,}
\nc{\para}{\mbox{\rm par}\,}
\nc{\aff}{\mbox{\rm aff}\,}
\nc{\crit}{\mbox{\rm crit}\,}
\nc{\cone}{\mbox{\rm cone}\,}
\nc{\dist}{\mbox{\rm dist}\,}
\nc{\diag}{\mbox{\rm Diag}\,}
\nc{\fix}{\mbox{\rm Fix}}
\nc{\rank}{\mbox{\rm rank}\,}

\newcommand{\lf}{\operatornamewithlimits{liminf}}

\newenvironment{myequation}{\begin{equation}}{\end{equation}}

\newenvironment{myeqnarray*}{\begin{eqnarray*}}{\end{eqnarray*}}
\nc{\bmye}{\begin{myequation}} \nc{\emye}{\end{myequation}}

\def\tto{\;{\lower 1pt \hbox{$\rightarrow$}}\kern -12pt
           \hbox{\raise 2.8pt \hbox{$\rightarrow$}}\;}

\usepackage{enumerate}
\usepackage{graphicx}

\usepackage{url}
\usepackage{cite}
\usepackage[margin=1.1in]{geometry}

\begin{document}

\author{D. Drusvyatskiy\thanks{%
    Department of Mathematics, University of Washington, Seattle, WA 98195-4350; Department of Combinatorics and Optimization, University of Waterloo, Waterloo, Ontario, Canada N2L 3G1; {\tt http://www.math.washington.edu/{\raise.17ex\hbox{$\scriptstyle\sim$}}ddrusv}; Research supported by AFOSR.   
    }%
	\and
  A.D. Ioffe\thanks{%
  Department of Mathematics, Technion-Israel Institute of Technology, Haifa, Israel 32000;
  {\tt ioffe@math.technion.ac.il}.
  Research supported in part by US-Israel Binational Science Foundation Grant 2008261.
}
	\and
  A.S. Lewis\thanks{%
  School of Operations Research and Information Engineering,
  Cornell University,
  Ithaca, New York, USA;
  {\tt http://people.orie.cornell.edu/{\raise.17ex\hbox{$\scriptstyle\sim$}}aslewis/}.
  Research supported in part by National Science Foundation Grant DMS-0806057 and by the US-Israel Binational Scientific Foundation Grant 2008261.
}
}

\title{\Large Generic minimizing behavior in semi-algebraic optimization}

\date{}
\maketitle

\abstract{We present a theorem of Sard type for semi-algebraic set-valued mappings whose graphs have dimension no larger than that of their range space: the inverse of such a mapping admits a single-valued analytic localization around any pair in the graph, for a generic value parameter.
This simple result yields a transparent and unified treatment of generic properties of semi-algebraic optimization problems: ``typical'' semi-algebraic problems have finitely many critical points, around each of which they admit a unique ``active manifold'' (analogue of an active set in nonlinear optimization); moreover, such critical points satisfy strict complementarity and second-order sufficient conditions for optimality are indeed necessary. 
}

\section{Introduction}
Many problems of contemporary interest can broadly be phrased as an inverse problem: given a vector $\bar{y}$ in $\R^m$ find a point $\bar{x}$ satisfying the inclusion
$$\bar{y}\in F(\bar{x}),$$
where $F\colon\R^n\rightrightarrows\R^m$ is some set-valued mapping (a mapping taking elements of $\R^n$ to subsets of $\R^m$) arising from the problem at hand.
In other words, we would like to find a point $\bar{x}$ such that the pair $(\bar{x},\bar{y})$ lies in the graph $$\gph F:=\{(x,y): y\in F(x)\}.$$ 
Stability analysis of such problems then revolves around understanding sensitivity of the solution set $F^{-1}(\bar{y})$ near $\bar{x}$ to small perturbations in $\bar{y}$. 
An extremely desirable property is for $F$ to be {\em strongly regular} \cite[Section 3G]{imp} at a pair $(\bar{x},\bar{y})$ in $\gph F$, meaning that
the graph of the inverse $F^{-1}$ coincides locally around $(\bar{y},\bar{x})$ with the graph of a single-valued Lipschitz continuous function $g\colon \R^m\to\R^n$. 
Naturally, then vectors $\bar y$ for which there exists a solution $\bar{x}\in F^{-1}(\bar{y})$ so that $F$ is not strongly regular at $(\bar{x},\bar{y})$ are called {\em weak critical values} of $F$.
We begin this work by asking the following question of Sard type: 
\begin{center}
Which mappings $F\colon\R^n\rightrightarrows\R^m$ have ``almost no'' weak critical values?
\end{center}
Little thought shows an immediate obstruction: the size of the graph of $F$.  
Clearly if $\gph F\subset\R^n\times\R^m$ has dimension (in some appropriate sense) larger than $m$, then no such result is possible. Hence, at the very least, we should insist that $\gph F$ is in some sense small in the ambient space $\R^n\times\R^m$.

Luckily, set-valued mappings having small graphs are common in optimization and variational analysis literature. Monotone operators make up a fundamental example: a mapping $F\colon\R^n\rightrightarrows\R^n$ is {\em monotone} if the inequality $\langle x_1-x_2, y_1-y_2 \rangle \geq 0$ holds whenever the pairs $(x_i,y_i)$ lie in $\gph F$. Minty \cite{minty} famously showed that the graph of a maximal monotone mapping on $\R^n$ is Lipschitz homeomorphic to $\R^n$, and hence monotone graphs can be considered small for our purposes. This property, for example, is fundamentally used in \cite{Eq_mon,newt_mon}. The most important example of monotone mappings in optimization is the subdifferential $\partial f$ of a convex function $f$. 
More generally, we may consider set-valued mapping arising from variational inequalities:
$$x\mapsto g(x)+N_Q(x),$$
where $g$ is locally Lipschitz continuous and $N_Q$ is the normal cone to a closed convex subset $Q$ of $\R^n$. Such mappings appear naturally in perturbation theory for variational inequalities; see \cite{imp}. One can easily check that the graph of this mapping is locally Lipschitz homeomorphic to $\gph N_Q$, and is therefore small in our understanding. In particular, we may look at conic optimization problems of the form
$$\min_x \{f(x): G(x)\in K\},$$
for a smooth function $f\colon\R^n\to\R$, a smooth mapping $G\colon\R^n\to\R^m$, and a closed convex cone $K$ in $\R^m$. Standard first order optimality conditions (under an appropriate qualification condition) amount to the variational inequality
$$\begin{bmatrix}
0 \\
0
\end{bmatrix}\in\begin{bmatrix}
\nabla f(x)+\nabla G(x)^*\lambda \\
-G(x)
\end{bmatrix}+ N_{\{0\}^n\times K^*}(x,\lambda),
$$
where $K^*$ is the dual cone of $K$ and the vector $\lambda$ serves as a generalized Lagrange multiplier; see \cite{imp} for a discussion. Consequently the set-valued mapping on the right-hand-side again has a small graph.

In summary, set-valued mappings with small graphs appear often, and naturally so, in optimization problems. 
Somewhat surprisingly, assuming that the graph is small is by itself not enough to guarantee that strong regularity is typical ---  the conclusion that we seek. For instance, there exists a $C^1$-smooth convex function $g\colon\R\to\R$ so that every number on the real line is a weakly critical value of the subdifferential $\partial g$. Such a function is easy to construct. Indeed, let $f\colon\R\to\R$ be a surjective, continuous, and strictly increasing function whose derivative is zero almost everywhere (such a function $f$ is described in \cite{sing} for example). Observe that $f$ is nowhere locally Lipschitz continuous, since otherwise the fundamental theorem of calculus would imply that that $f$ is constant on some interval --- a contradiction.  On the other hand $f$ is the derivative of the function $h(t):=\int^t_{0} f(r) \,dr$. The Fenchel conjugate $h^*\colon\R\to\R$ is then exactly the function $g$ that we seek. This example is interesting in light of Mignot's theorem \cite[Theorem 9.65]{VA}, which guarantees that at almost every subgradient, the inverse of the convex subdifferential must be single-valued and differentiable, though as we see, not necessarily locally Lipschitz continuous. 

In light of this example, we see that even monotone variational inequalities can generically fail to be strongly regular. Incidentally, this explains the absence of Sard's theorem from all standard texts on variational inequalities (e.g. \cite{imp, fac_pang, fin_2}), thereby deviating from classical mathematical analysis literature where implicit function theorems go hand in hand with Sard's theorem. 

Motivated by optimization problems typically arising in practice, we consider {\em semi-algebraic} set-valued mappings --- those whose graphs can be written as a finite union of sets each defined by finitely many polynomial inequalities. See for example \cite{tame_opt} on the role of such mappings in nonsmooth optimization. In Theorem~\ref{thm:sard}, we observe that any semi-algebraic mapping  $F\colon\R^n\rightrightarrows\R^m$, whose graph has dimension no larger than $m$, has almost no weak critical values (in the sense of Lebesgue measure). Thus in the semi-algebraic setting, the size of the graph is the only obstruction to the Sard-type theorem that we seek. 

Despite its simplicity, both in the statement and the proof, Theorem~\ref{thm:sard} leads to a transparent and unified treatment of generic properties of semi-algebraic optimization problems, covering in particular polynomial optimization problems, semi-definite programming, and copositive optimization --- topics of contemporary interest. To illustrate, consider the family of optimization problems
$$\min_x f(x)+h(G(x)+y)-v^Tx,$$
where $f$ and $h$ are semi-algebraic functions on $\R^n$ and $\R^m$, respectively, and $G\colon\R^n\to\R^m$ is a $C^2$-smooth semi-algebraic mapping. Here the vectors $v,y$ serve as perturbation parameters. First order optimality conditions (under an appropriate qualification condition) then take the form of a generalized equation
$$\begin{bmatrix}
v \\
y
\end{bmatrix}\in\begin{bmatrix}
\nabla G(x)^*\lambda \\
-G(x)
\end{bmatrix}+ \Big(\partial f\times (\partial h)^{-1}\Big)(x,\lambda),
$$
where the subdifferentials $\partial f$ and $\partial h$ are meant in the limiting sense; see e.g. \cite{VA}.
Observe that the perturbation parameters $(v,y)$ appear in the range of the set-valued mapping on the right-hand-side. This set-valued mapping in turn, has a small graph. Indeed, the graphs of the subdifferential mappings $\partial f$  and $\partial h$ always have dimension exactly $n$ and $m$, respectively \cite[Theorem 3.7]{dim} (even locally around each of their points \cite[Theorem 3.8]{loc}, \cite[Theorem 5.13]{dir_lip}); monotonicity or convexity are irrelevant here. Thus the semi-algebraic Sard's theorem applies. In turn, appealing to some standard semi-algebraic techniques, we immediately conclude:
for almost all parameters $(v,y)\in\R^{n}\times \R^m$, the problem admits finitely many composite critical points with each one satisfying a strict complementarity condition, a basic qualification condition (generalizing that of Mangasarian-Fromovitz) holds, both $f$ and $h$ admit unique active manifolds in the sense of \cite{ident, Lewis-active}, and positivity of a second-derivative (of parabolic type) is both necessary and sufficient for second-order growth.

This development nicely unifies and complements a number of earlier results, such as the papers \cite{gen_spin,spin} on generic optimality conditions in nonlinear programming, the study of the complementarity problem \cite{saigal}, generic strict complementarity and nondegeneracy in semi-definite programming \cite{aliz, shap_non}, as well as the general study of strict complementarity in convex optimization  \cite{pat_tun,gen_nondeg}. In contrast, many of our arguments are entirely 
independent of the representation of the semi-algebraic optimization problem at hand. 
It is worth noting that convexity (and even Clarke regularity) is of no consequence for us. In particular, our results generalize and drastically simplify the main results of \cite{gen}, where convexity of the semi-algebraic optimization problem plays a key role. Though we state our results for semi-algebraic problems, they all generalize to the ``tame'' setting; see \cite{tame_opt} for the definitions. Key elements of the development we present here were first reported in \cite{adrian_ICM}.  In particular Theorem 7.3 in that work sketches the proof of generic minimizing behavior, restricted for simplicity to the case of linear optimization over closed semi-algebraic sets.

The outline of the manuscript is as follows. We begin in Section~\ref{sec:notation}, by recording some basic notation to be used throughout the manuscript. In Section~\ref{sec:Sard}, we recall some rudimentary elements of semi-algebraic geometry and prove the semi-algebraic Sard theorem for weak critical values. In Section~\ref{sec: semi_func_gen}, we establish various critical point properties of generic semi-algebraic functions, while in Section~\ref{sec:Lag}, we refine the analysis of the previous section for semi-algebraic functions in composite form.

\section{Basic notation}\label{sec:notation}
We begin by summarizing a few basic notions of variational and set-valued analysis. Unless otherwise stated, we follow the terminology and notation of \cite{VA,imp}. Throughout $\R^n$ will denote an $n$-dimensional Euclidean space with inner-product $\langle \cdot, \cdot\rangle$ and corresponding norm $|\cdot|$. We denote by $B_{\epsilon}(x)$ an open ball of radius $\epsilon$ around a point $x$ in $\R^n$.

A set-valued mapping $F$ from $\R^n$ to $\R^m$, denoted $F\colon\R^n\rightrightarrows\R^m$, is a mapping taking points in $\R^n$ to subsets of $\R^m$,
with the {\em domain} and {\em graph} of $F$ being
$$\dom F:=\{x\in\R^n: F(x)\neq \emptyset\},$$
$$\gph F:=\{(x,y)\in\R^n\times\R^m: y\in F(x)\}.$$
We say that $F$ is {\em finite-valued}, when the cardinality of the image $F(x)$ is finite for every $x\in\R^n$. 

A mapping $\hat{F}\colon\R^n\rightrightarrows\R^m$ is a {\em localization} of $F$ around $(\bar{x},\bar{y})\in\gph F$ if the graphs of $F$ and $\hat{F}$ coincide on a neighborhood of $(\bar{x},\bar{y})$. The following is the central notion we explore.

\begin{defn}[Strong regularity and weak critical points]
{\rm 
A set-valued mapping $F\colon\R^n\rightrightarrows\R^m$ 
is $C^p$-{\em strongly regular} at $(\bar{x},\bar{y})\in\gph F$ if the inverse $F^{-1}$ admits a $C^p$-smooth single-valued localization around  $(\bar{y},\bar{x})$. 

A vector $\bar{y}\in\R^m$ is a $C^p$-{\em weak critical value} of $F$ if there exists a point $\bar{x}$ in the preimage $F^{-1}(\bar{y})$, so that $F$ is not $C^p$-strongly regular at $(\bar{x},\bar{y})$.
}
\end{defn}

Observe that $\bar{y}$ being a weak critical value of $F$, at the very least, entails that the preimage $F^{-1}(\bar{y})$ is nonempty. It is instructive to comment on the terms ``strong'' and ``weak''. We use these to differentiate strong regularity from the weaker notion of metric regularity \cite{imp, ioffe_survey} and the corresponding criticality concept. Note that the term ``weakly critical'' (with no qualifier) refers to the real-analytic version of the definition.

A mapping $F\colon Q\to\widetilde{Q}$, where $\widetilde{Q}$ is a subset of $\R^m$, is ${C}^p$-{\em
smooth} if for each point $\bar{x}\in Q$, there is a neighborhood
$U$ of $\bar{x}$ and a ${C}^p$-smooth mapping $\widehat{F}\colon
\R^n\to\R^m$ that agrees with $F$ on $Q\cap U$. The symbol $C^{\omega}$ will always mean real analytic. Smooth manifolds will play an important role in our work; a nice reference is \cite{lee}.
\begin{defn}[Smooth manifolds]
{\rm
A subset $\mathcal{M}\subset\R^n$, is a $C^p$ {\em manifold of dimension} $r$ if for each point $\bar{x}\in \mathcal{M}$, there is an open neighborhood $U$ around $\bar{x}$ and a mapping $F$ from $\R^n$ to a $(n-r)$-dimensional Euclidean space so that $F$ is $C^p$-smooth with the derivative $\nabla F(\bar{x})$ having full rank and we have $$\mathcal{M}\cap U=\{x\in U: F(x)=0\}.$$ In this case, the {\em tangent space} to $\mathcal{M}$ at $\bar{x}$ is simply the set $T_{\mathcal{M}}(\bar{x}):=\ker \nabla F(\bar{x})$, while the {\em normal space} to $\mathcal{M}$ at $\bar{x}$ is defined by $N_{\mathcal{M}}(\bar{x}):=\range \nabla F(\bar{x})^*$.}
\end{defn}

Given a $C^1$-smooth manifold $\mathcal{M}$ and a mapping $F$ that is $C^1$-smooth on $\mathcal{M}$, we will say that $F$ has {\em constant rank on} $\mathcal{M}$ if the rank of the operator $\nabla \widehat{F}(x)$ restricted to $T_{\mathcal{M}}(x)$, with $\widehat{F}$ being any $C^1$-smooth mapping agreeing with $F$ on a neighborhood of $x$ in $\mathcal{M}$, is the same for all $x\in \mathcal{M}$.


\section{Semi-algebraic geometry and Sard's theorem}\label{sec:Sard}
Our current work is cast in the setting of semi-algebraic geometry. A {\em semi-algebraic} set $Q\subset\R^n$ is a finite union of sets of the form $$\{x\in \R^n: P_1(x)=0,\ldots,P_k(x)=0, R_1(x)<0,\ldots, R_l(x)<0\},$$ where $P_1,\ldots,P_k$ and $R_1,\ldots,R_l$ are polynomials in $n$ variables. In other words, $Q$ is a union of finitely many sets, each defined by finitely many polynomial equalities and inequalities. A map $F\colon\R^n\rightrightarrows\R^m$ is {\em semi-algebraic} if $\mbox{\rm gph}\, F\subset\R^{n+m}$ is a semi-algebraic set. 
For more details on semi-algebraic geometry, see for example \cite{Coste-semi,DM}.  An important feature of semi-algebraic sets is that they can be decomposed into analytic manifolds. Imposing a very weak condition on the way the manifolds fit together, we arrive at the following notion.

\begin{defn}[Stratification]
{\rm
A $C^p$-{\em stratification} of a semi-algebraic set $Q$ is a finite partition of $Q$ into disjoint semi-algebraic $C^p$ manifolds
$\{\mathcal{M}_i\}$ (called {\em strata}) with the property that for each index $i$, the intersection of the
closure of $\mathcal{M}_i$ with $Q$ is the union of some $\mathcal{M}_j$'s.}
\end{defn}

In particular, we can now define the dimension of any semi-algebraic set $Q$.
\begin{defn}[Dimension of semi-algebraic sets]{\hfill \\ }
{\rm
The {\em dimension} of a semi-algebraic set $Q\subset\R^n$ is the maximal dimension of a semi-algebraic $C^1$ manifold appearing in any $C^1$-stratification of $Q$.}
\end{defn}

It turns out that the dimension of a semi-algebraic set $Q$ does not depend on any particular stratification. 
It is often useful to refine stratifications. Consequently, the following notation becomes convenient.
\begin{defn}[Compatibility]
{\rm Given finite collections $\{B_i\}$ and $\{C_j\}$ of subsets of $\R^n$, we say that $\{B_i\}$ is {\em compatible} with $\{C_j\}$ if for all $B_i$ and $C_j$, either $B_i\cap C_j=\emptyset$ or $B_i\subset C_j$.}
\end{defn}

As we have alluded to at the onset, the following is a deep existence theorem for semi-algebraic stratifications \cite[Theorem 4.8]{DM}.
\begin{thm}[Stratifications exist]\label{thm:strat_exist}
Consider a semi-algebraic set $Q$ in $\R^n$ and a semi-algebraic map $F\colon Q\to \R^m$. Let $\mathcal{A}$ be a finite collection of semi-algebraic subsets of $Q$ and $\mathcal{B}$ a finite collection of semi-algebraic subsets of $\R^m$. Then there exists a $C^{\omega}$-stratification $\mathcal{A}'$ of $Q$ that is compatible with $\mathcal{A}$ and a $C^{\omega}$-stratification $\mathcal{B}'$ of $\R^m$ compatible with $\mathcal{B}$ such that for every stratum $\mathcal{M}\in\mathcal{A}'$, the restriction of $F$ to ${\mathcal{M}}$ is analytic and has constant rank, and the image
$F(\mathcal{M})$ is a stratum in $\mathcal{B}'$.
\end{thm}

Classically a set $U\subset\R^n$ is said to be ``generic'', if it is large in some precise mathematical sense, depending on context. Two popular choices are that of $U$ being full-measure, meaning its complement has Lebesgue measure zero, and that of $U$ being topologically generic, meaning it contains a countable intersection of dense open sets. In general, these notions are very different. However for semi-algebraic sets, the situation simplifies drastically. Indeed, if $U\subset\R^n$ is a semi-algebraic set, then the following are equivalent.
\begin{itemize}
\item $U$ is dense.
\item $U$ is full-measure. 
\item $U$ is topologically generic. 
\item The dimension of $U^c$ is strictly smaller than $n$.
\end{itemize}
Complements of such sets are said to be {\em negligible}.



The following is the basic tool that we will use. A semi-algebraic finite-valued mapping $F\colon\R^n\rightrightarrows\R^m$ can be decomposed into finitely many $C^{\omega}$-smooth single-valued selections that ``cross'' almost nowhere. This result is standard: it readily follows for example from \cite[Corollary~2.27]{dim}. We provide a proof sketch for completeness.
\begin{thm}[Selections of finite-valued semi-algebraic mappings]\label{thm:selfin}
Consider a finite-valued semi-algebraic mapping $G\colon\R^n\rightrightarrows\R^m$. 
Then there exists an integer $N$, a finite collection of open semi-algebraic sets $\{U_i\}^N_{i=0}$ in $\R^n$, and analytic semi-algebraic single-valued mappings  $$G_i^j\colon U_i\to\R^m \qquad\textrm{ for } i=0,\ldots,N \textrm{ and } j=1,\ldots, i$$ satisfying:
\begin{enumerate}
\item $\bigcup_i U_i$ is dense in $\R^n$.
\item For any $x\in U_i$, the image $G(x)$ has cardinality $i$.
\item We have the representation
$$G(x)  =  \{G_i^j(x) : j=1,2,\ldots,i\} \qquad\textrm{ whenever }x\in U_i.$$
\end{enumerate}
\end{thm}
\begin{proof}
Since $G$ is semi-algebraic, there exists an integer $N$ with the property that the cardinality of the images $G(x)$ is no greater than $N$ \cite[Theorem 4.4]{DM}.
For $i=0,\ldots, k$, define $U_i$ to be the set of points $x\in\R^n$ so that that image $G(x)$ has cardinality precisely equal to $i$. A standard argument shows that the sets $U_i$ are semi-algebraic. Stratifying, we replace each $U_i$ with an open set (possibly empty) so that the union of $U_i$ is dense in $\R^n$. 

Fix now an index $i$. By \cite[Corollary 2.27]{dim}, there exists a dense open subset $X_i$ of $U_i$ with the property that there exists a semi-algebraic set $Y_i\subset \R^m$ and a semi-algebraic homeomorphism $\theta_i\colon\gph G\big|_{X_i}\to X_i\times Y_i$ satisfying
$$\theta_i(\{x\}\times G(x))=\{x\}\times Y_i \quad \textrm{ for all }x\in X_i.$$ 
Observe that for each $i$ the set $Y_i$ has cardinality $i$. Enumerate the elements of $Y_i$ by labeling $Y_i=\{y_1,\ldots,y_i\}$.
Define $\pi$ to be the projection $\pi(x,y)=y$ and for each $j=1,\ldots,i$ set $$G^j_{i}(x)=\pi\circ\theta^{-1}_i(x,y_j)\quad \textrm{ for } x\in X_i.$$
Stratifying $X_i$, we may replace $U_i$ by an open dense subset on which all the mappings $G^j_{i}$ are analytic. The result follows.
\end{proof}

In particular, this theorem is applicable for semi-algebraic mappings with ``small'' graphs, since such mappings are finite-valued almost everywhere \cite[Proposition~4.3]{dim}.
\begin{thm}[Finite selections for mappings with small graphs]\label{thm:sel}
Suppose that the graph of a semi-algebraic set-valued mapping $F\colon\R^n\rightrightarrows\R^m$ has dimension no larger than $m$. 
Then the inverse mapping $F^{-1}\colon\R^m\rightrightarrows\R^n$ is finite-valued almost everywhere. 
\end{thm}

We now arrive at the semi-algebraic Sard theorem --- the main result of this section. 
\begin{thm}[Semi-algebraic Sard theorem for weakly critical values]\label{thm:sard}
Consider a semi-algebraic set-valued mapping $F\colon\R^n\rightrightarrows\R^m$ satisfying $\dim \gph F\leq m$. Then the collection of weakly critical values of $F$ is a negligible semi-algebraic set. More precisely, 
there exists an integer $N$, a finite collection of open semi-algebraic sets $\{U_i\}^N_{i=0}$ in $\R^n$, and analytic semi-algebraic single-valued mappings  $$G_i^j\colon U_i\to\R^n \qquad\textrm{ for } i=0,\ldots,N \textrm{ and } j=1,\ldots, i$$ satisfying:
\begin{enumerate}
\item $\bigcup_i U_i$ is dense in $\R^m$.
\item For any $x\in U_i$, the preimage $F^{-1}(x)$ has cardinality $i$.
\item We have the representation
$$F^{-1}(x)  =  \{G_i^j(x) : j=1,2,\ldots,i\} \qquad\textrm{ whenever }x\in U_i.$$
\end{enumerate}
\end{thm}
\begin{proof}
Consider the open semi-algebraic sets $\{U_i\}^N_{i=0}$ along with the single-valued, analytic, semi-algebraic mappings  $G_i^j\colon U_i\to\R^n$ provided by Theorems~\ref{thm:selfin} and ~\ref{thm:sel}. Since for any $y\in U_i$, the preimage $F^{-1}(y)$ has cardinality $i$ and we have $F^{-1}(y)  =  \{G_i^j(y) : j=1,2,\ldots,i\}$, we deduce that the values $G_i^j(y)$ for $j=1,\ldots,i$ are all distinct. Since $G_i^j$ are in particular continuous, 
we deduce that the mapping $F^{-1}$  has a single-valued analytic localization around $(y,x)$ for every point $x\in F^{-1}(y)$.
The result follows.
\end{proof}

We note that a Sard type theorem for semi-algebraic set-valued mapping with possibly large graphs, where criticality means absence of ``metric regularity'' \cite{imp, ioffe_survey}, was proved in \cite{IS}. Since we will not use this concept in the current work, we omit the details. 

\section{Critical points of generic semi-algebraic functions}\label{sec: semi_func_gen}
In this section, we derive properties of critical points (appropriately defined) of semi-algebraic functions under generic linear perturbations. Throughout, we will consider functions $f$ on $\R^n$ taking values in the extended-real-line $\overline{\R}=\R\cup\{+\infty\}$. We will always assume that such functions are proper, meaning they are not identically equal to $+\infty$. The {\em domain} and {\em epigraph} of $f$ are 
\begin{align*}
\dom f &:= \{x\in\R^n: f(x) <+\infty\},\\
\epi f &:= \{(x,r)\in\R^n\times\R: r \geq f(x)\}.
\end{align*}
The {\em indicator function} of a set $Q\subset \R^n$, denoted $\delta_Q$, is defined to be zero on $Q$ and $+\infty$ off it. 
A function $f\colon\R^n\to\overline{\R}$ is {\em lower-semicontinuous} ({\em lsc}) whenever the epigraph $\epi f$ is closed. 
The notion of criticality we consider arises from the workhorse of variation analysis, the subdifferential. 
\begin{defn}[Subdifferentials and critical points]
{\rm Consider a function $f\colon\R^n\to\overline{\R}$ and a point $\bar{x}$ with $f(\bar{x})$ finite. 
\begin{enumerate}
\item The {\em proximal subdifferential} of $f$ at $\bar{x}$, denoted 
$\partial_p f(\bar{x})$, consists of all vectors $v \in \R^n$ satisfying $$f(x)\geq f(\bar{x})+\langle v,x-\bar{x} \rangle +O(|x-\bar{x}|^2).$$ 
\item The {\em limiting subdifferential} of $f$ at $\bar{x}$, denoted $\partial f(\bar{x})$, consists of all vectors $v\in\R^n$ for which there exist sequences $x_i\in\R^n$ and $v_i\in\partial_p  f(x_i)$ with $(x_i,f(x_i),v_i)$ converging to $(\bar{x},f(\bar{x}),v)$.
\item The {\em horizon subdifferential} of $f$ at $\bar{x}$, denoted $\partial^{\infty} f(\bar{x})$, consists of all vectors $v\in\R^n$ for which there exist points $x_i\in\R^n$, vectors $v_i\in\partial  f(x_i)$, and real numbers $t_i\searrow 0$ with $(x_i,f(x_i),t_iv_i)$ converging to $(\bar{x},f(\bar{x}),v)$.

\end{enumerate}
We say that $\bar{x}$ is a {\em critical point} of $f$ whenever the inclusion $0\in\partial f(\bar{x})$ holds.}
\end{defn}


The subdifferentials $\partial_p f$ and $\partial f$ generalize the notion of a gradient to the nonsmooth setting. In particular, if $f$ is $C^2$-smooth, then $\partial_P f$ and $\partial f$ simply coincide with the gradient $\nabla f$, while if $f$ is convex, both subdifferentials coincide with the subdifferential of convex analysis \cite[Proposition~8.12]{VA}. The horizon subdifferential $\partial^{\infty} f$ plays an entirely different role: it detects horizontal normals to the epigraph of $f$ and is instrumental in establishing calculus rules \cite[Theorem 10.6]{VA}.
For any set $Q\subset\R^n$, we define the {\em proximal} and {\em limiting normal cones} by the formulas $N^p_Q:=\partial_p \delta_Q$ and $N_Q:=\partial \delta_Q$, respectively.

We will show in this section that any semi-algebraic function, subject to a generic linear perturbation, satisfies a number of desirable properties around any of its critical points. To this end, a key result for us will be that whenever  $f\colon\R^n\to\overline{\R}$ is semi-algebraic, the graphs of the two subdifferentials $\partial_p f$ and $\partial f$ have dimension exactly $n$ \cite[Theorem 3.7]{dim}. (This remains true even in a local sense within the subdifferential graphs \cite[Theorem 3.8]{loc}, \cite[Theorem 5.13]{dir_lip}).
Combining this with Theorem~\ref{thm:sard}, we immediately deduce that generic subgradients of a semi-algebraic function are not weakly critical.


This observation, in turn, has immediate implications for minimizers of generic semi-algebraic functions, since strong regularity of the subdifferential is closely related to quadratic growth of the function. To be more precise, recall that $\bar{x}$ is a {\em strong local minimizer} of a function $f$ whenever there exist $\alpha >0$ and a neighborhood $U$ of $\bar{x}$ so that 
$$f(x)\geq f(\bar{x})+\frac{\alpha}{2} |x-\bar{x}|^2 \qquad \textrm{ for each } x \textrm{ in } U.$$
A more stable version of this condition follows.
\begin{defn}[Stable strong local minimizers]
{\rm A point $\bar{x}$ is a {\em stable strong local minimizer}\footnote{\rm
This notion appears under the name of uniform quadratic growth for tilt perturbations in \cite{Bon_Shap}, where it is considered in the context of optimization problems in composite form.} of a function $f\colon\R^n\to\overline{\R}$ if
there exist $\alpha >0$ and a neighborhood $U$ of $\bar{x}$ so that for every vector $v$ near the origin, there is a point $x_v$ (necessarily unique) in $U$, with $x_0=\bar{x}$, so that in terms of the perturbed functions $f_v:=f(\cdot)-\langle v,\cdot\rangle$, the inequality 
\begin{equation*}
f_{v}(x)\geq f_{v}(x_{v})+\frac{\alpha}{2} |x-x_v|^2 \qquad \textrm{ holds for each } x \textrm{ in } U.
\end{equation*}}
\end{defn}

In \cite[Proposition~3.1, Corollary~3.2]{tilt}, the authors show that strong metric regularity of the subdifferential at $(x, v)$, where $x$ is a local  minimizer of $f_v:=f(\cdot)-\langle v,\cdot\rangle$, always implies that $x$ is a stable strong local minimizer of $f_v$. See also \cite{crit_semi, tilt_other} for related results. Thus local minimizers of any semi-algebraic function, for a generic linear perturbation parameter, are stable strong local minimizers.
Moreover, since the subdifferentials all have dimension exactly $n$ and $\dim (\gph \partial f)\setminus (\gph \partial_p f)\leq n$ it easy to see that for a generic vector $v$, the strict complementarity condition 
$$v\in \partial f(x)\quad \Longrightarrow \quad v\in\ri \partial_p f(x) \quad\textrm{ holds for any }x\in\R^n.$$

We summarize all of these observations below.
\begin{cor}[Basic generic properties of semi-algebraic problems]\label{cor:basic}
Consider an lsc, semi-algebraic function $f\colon\R^n\to\overline{\R}$. Then there exists an integer $N >0$ such that for a generic vector $v\in\R^n$ the function 
$$f_v(x):=f(x)-\langle v,x\rangle$$
has no more than $N$ critical points. In turn, each such critical point $\bar{x}$ satisfies the strict complementarity condition 
$$0\in \ri\, \partial_p f_v(\bar{x}),$$
and if moreover $\bar{x}$ is a local minimizer of $f_v$, then $\bar{x}$ is a stable strong local minimizer.
\end{cor}

We will see that by appealing further to semi-algebraic stratifications much more is true: any semi-algebraic function, up to a generic perturbation, admits a unique ``stable active set''. To introduce this notion, we briefly record some notation.
To this end, working with possibly discontinuous functions $f\colon\R^n\to\overline{\R}$, it is useful to consider $f${\em-attentive} convergence of a sequence $x_i$ to a point $\bar{x}$, denoted $x_i \xrightarrow[f]{} \bar{x}$.
In this notation 
$$x_i \xrightarrow[f]{} \bar{x} \quad\Longleftrightarrow\quad x_i\to\bar{x} \textrm{ and } f(x_i)\to f(\bar{x}).$$
An $f$-{\em attentive localization} of $\partial f$ at $(\bar{x},\bar{v})$ is any mapping $T\colon\R^n\rightrightarrows\R^n$ that coincides on an $f$-attentive neighborhood of $\bar{x}$ with some localization of $\partial f$ at $(\bar{x},\bar{v})$.

It is often useful to require a kind of uniformity of subgradients. Recall that the subdifferential $\partial f$ of an lsc convex function $f$ is {\em monotone} in the sense that $\langle v_1-v_2,x_1-x_2\rangle \geq 0$ for any pairs $(x_1,v_1)$ and $(x_2,v_2)$ in $\gph \partial f$. Relaxing this property slightly leads to the following concept \cite[Definition~1.1]{prox_reg}. 
\smallskip
\begin{defn}[Prox-regularity]
{\rm An lsc function $f\colon\R^n\to\overline{\R}$ is called {\em
prox-regular at} $\bar{x}$ {\em for} $\bar{v}$, with
$\bar{v}\in\partial_p f(\bar{x})$, if there exists a constant $r >0$ and an $f$-attentive localization $T$ of $\partial f$ around $(\bar{x},\bar{v})$ so that $T+rI$ is monotone.
}
\end{defn}
\smallskip

In particular ${C}^2$-smooth functions and lsc, convex functions are prox-regular at each of their points \cite[Example 13.30, Proposition 13.34]{VA}. In contrast, the negative norm function $x\mapsto -|x|$ is not prox-regular at the origin.

We are now ready to state what we mean by a ``stable active set''. This notion introduced in \cite{Lewis-active}, and rooted in even earlier manuscripts \cite{Wright,Al-Khayyal-Kyparisis91,Burke-More88, Burke90,Calamai-More87,Dunn87,Ferris91,Flam92}, extends active sets in nonlinear programming far beyond the classical setting. The exact details of the definition will not be important for us, since we will immediately pass to an equivalent, but more convenient for our purposes, companion concept. Roughly speaking, a smooth manifold $\mathcal{M}$ is said to be ``active'' or ``partly smooth'' for a function $f$ whenever $f$ varies smoothly along the manifold and sharply off it. The {\em parallel subspace} of any nonempty set $Q$, denoted $\para Q$, is the affine hull of $\conv\, Q$ translated to contain the origin. We also adopt the convention $\para \emptyset=\emptyset$.
\begin{defn}[Partial smoothness]
{\rm Consider an lsc function $f\colon\R^n\to\overline{\R}$ and a $C^p$ manifold $\mathcal{M}$. Then $f$ is $C^p${\em -partly smooth}  ($p \geq 2$) with respect to $\mathcal{M}$ at $\bar{x}\in\mathcal{M}$ for $\bar{v}\in  \partial f(\bar{x})$ if
\begin{enumerate}
\item {\bf (smoothness)} $f$ restricted to $\mathcal{M}$ is $C^p$-smooth on a neighborhood of $\bar{x}$. 
\item {\bf (prox-regularity)} $f$ is prox-regular at $\bar{x}$ for $\bar{v}$.
\item {\bf (sharpness)} $\para \partial_p f(\bar{x})= N_{\mathcal{M}}(\bar{x})$.
\item {\bf (continuity)} There exists a  neighborhood $V$ of $\bar{v}$, such that the mapping,
$x\mapsto V\cap \partial f(\bar{x})$, when restricted to $\mathcal{M}$, is inner-semicontinuous at $\bar{x}$.
\end{enumerate}
}
\end{defn}

In \cite[Proposition~8.4]{ident}, it was shown that the somewhat involved definition of partial smoothness can be captured more succinctly, assuming a strict complementarity condition $\bar{v}\in\ri \partial_p f(\bar{x})$. Indeed, the essence of partial smoothness is in the fact that algorithms generating iterates, along with approximate criticality certificates, often ``identify'' a distinguished manifold in finitely many iterations; see the extensive discussions in \cite{ident_active,ident}. 

\begin{defn}[Identifiable manifolds]
{\rm Consider an lsc function $f\colon\R^n\to\overline{\R}$. Then a set $\mathcal{M}\subset\R^n$ is a $C^p$ {\em identifiable manifold of }$f\colon\R^n\to\overline{\R}$ {\em at a point} $\bar{x}\in \mathcal{M}$ {\em for} $\bar{v}\in\partial f(\bar{x})$ if the set $\mathcal{M}$ is a $C^p$ manifold around $\bar{x}$, the restriction of $f$ to $\mathcal{M}$ is $C^p$-smooth around $\bar{x}$, and $\mathcal{M}$ has the finite identification property: for any sequences $x_i\xrightarrow[f]{}\bar{x}$ and $v_i\to\bar{v}$, with $v_i\in\partial f(x_i)$, the points $x_i$ must lie in $\mathcal{M}$ for all sufficiently large indices $i$.
}
\end{defn}

In \cite[Proposition~8.4]{ident}, the authors showed that for $p\geq 2$ the two sophisticated looking properties
\begin{enumerate}
\item $f$ is $C^p$-partly smooth with respect to $\mathcal{M}$ at $\bar{x}$ for $\bar{v}$,
\item $\bar{v}\in\ri \partial_p f(\bar{x})$,
\end{enumerate}
taken together are simply equivalent to $\mathcal{M}$ being a $C^p$ identifiable manifold of $f$ at $\bar{x}$ for $\bar{v}\in\partial_p f(\bar{x})$. This will be the key observation that we will use with regard to partly smooth manifolds. 

It is important to note that identifiable manifolds can fail to exist. For example, the function $f(x,y)=(|x|+|y|)^2$ does not admit any identifiable manifold at the origin for the zero subgradient. On the other hand, we will see that such behavior, in a precise mathematical sense, is rare.

Roughly speaking, existence of an identifiable manifold at a critical point opens the door to Newton-type acceleration strategies \cite{ident_active,dual_av,MC05} and moreover certifies that sensitivity analysis of the nonsmooth problem is in essence classical \cite{shanshan,Hare}. To illustrate, we record two basic properties of identifiable manifolds \cite[Propositions~5.9, 7.2]{ident}, which we will use in Section~\ref{sec:Lag}. 

\begin{thm}[Basic properties of identifiable manifolds]\label{thm:basic_prop_ident}
Consider an lsc function $f\colon\R^n\to\overline{\R}$ and suppose that $\mathcal{M}$ is a $C^2$-identifiable manifold around $\bar{x}$ for $\bar{v}=0\in \partial_p f(\bar{x})$. 
Then the following are equivalent
\begin{enumerate} 
\item $\bar{x}$ is a strong local minimizer of $f$.
\item $\bar{x}$ is a strong local minimizer of $f+\delta_{\mathcal{M}}$.
\end{enumerate}
Moreover, equality
$$\gph \partial f= \gph \partial (f+\delta_{\mathcal{M}}),$$
holds on an f-attentive neighborhood of $(\bar{x},\bar{v})$.

\end{thm}



Generic existence of identifiable manifolds for semi-algebraic functions will now be a simple consequence of stratifiability of semi-algebraic sets. 
We note that, in particular, it shows that convexity is superfluous for the main results of \cite{gen}. 
\begin{cor}[Generic properties of semi-algebraic problems]\label{cor:basic2}
Consider an lsc, semi-algebraic function $f\colon\R^n\to\overline{\R}$. Then there exists an integer $N >0$ such that for a generic vector $v\in\R^n$ the function 
$$f_v(x):=f(x)-\langle v,x\rangle$$
has no more than $N$ critical points. Moreover each such critical point $\bar{x}$ satisfies 
\begin{enumerate}
\item\label{item:prox} {\bf (prox-regularity)} $f_v$ is prox-regular at $\bar{x}$ for $0$.
\item\label{item:strict} {\bf (strict complementarity)} The inclusion $0\in \ri\, \partial_p f_v(\bar{x})$ holds.
\item\label{item:ident} {\bf (identifiable manifold)} $f_v$ admits a $C^\omega$ identifiable manifold at $\bar{x}$ for $0$.
\item\label{item:smoothdep_crit} {\bf (smooth dependence of critical points)} The subdifferential $\partial f$ is strongly regular at $(\bar{x},v)$. More precisely, 
there exist neighborhoods $U$ of $\bar{x}$ and $V$ of $v$ so that the critical point mapping
$$w\mapsto U\cap (\partial f)^{-1}(w)=\{x\in U: x \textrm{ is critical for } f(\cdot)-\langle w,\cdot\rangle\}$$
is single-valued and analytic on $V$, and maps $V$ onto $\mathcal{M}$.
\end{enumerate}
Moreover if $\bar{x}$ is a local minimizer of $f_v$ then $\bar{x}$ is in fact a {\bf stable strong local minimizer} of $f_v$.
\end{cor}
\begin{proof}
Generic finiteness of critical points and generic strict complementarity was already recorded in Corollary~\ref{cor:basic}. We now tackle existence of identifiable manifolds.
To this end, by \cite[Theorem~3.7]{dim}, the graph of the subdifferential mapping $\partial f\colon\R^n\rightrightarrows\R^n$ has dimension $n$.  
Consequently, applying Theorem~\ref{thm:sard}, we obtain a collection of open semi-algebraic sets $\{U_i\}^k_{i=0}$ of $\R^n$, with dense union, and analytic semi-algebraic single-valued mappings  $$G_i^j\colon U_i\to\R^n \qquad\textrm{ for } i=0,\ldots,k \textrm{ and } j=1,\ldots, i$$ with the property that for each $v\in U_i$ the set $(\partial f)^{-1}(v)$ has cardinality $i$ and we have the representation
$$(\partial f)^{-1}(v)  =  \{G_i^j(v) : j=1,2,\ldots,i\}.$$
Let $\mathcal{B}$ now be a stratification of $\dom f$ so that $f$ is analytic on each stratum.
Applying Theorem~\ref{thm:strat_exist} to each $G_i^j$, we obtain a stratification $\mathcal{A}^j_i$ of $U_i$ so that $G_i^j$ is analytic and has constant rank on each stratum of $\mathcal{M}$ of 
$\mathcal{A}^j_i$, and so that $f$ is analytic on the images $G_i^j(\mathcal{M})$. 
Finding a stratification of $U_i$ compatible with $\bigcup_j \mathcal{A}^j_i$, we obtain a dense open subset $\hat{U}_i$ of $U_i$ so that around each point $v\in \hat{U}_i$ there exists a neighborhood $V$ of $v$ so that $G_i^j$ is analytic and has constant rank on $V$, and so that $f$ is analytic on the images $G_i^j(V)$. 
Due to the constant rank condition, decreasing $V$ further, we may be assured that $G_i^j(V)$ are all analytic manifolds. 
Taking into account Theorem~\ref{thm:sard}, we may also assume that none of the values in $\hat{U}_i$ are weakly critical. 
Consequently for each $v\in \hat{U}_i$, there exists a sufficiently small neighborhood $V$ of  $v$ so that the analytic manifold $G_i^j(V)$ coincides with $(\partial f)^{-1}(V)$ on a  neighborhood of $G_i^j(v)$. Hence $G_i^j(V)$ is an identifiable manifold at  $G_i^j(v)$ for $v$. 
Finally, appealing to Corollary~\ref{cor:basic}, the result follows.
\end{proof}

Next we look more closely at second order growth, from the perspective of second derivatives. To this end, we record the following standard definition.

\begin{defn}[Subderivatives]
{\rm
Consider a function $f\colon\R^n\to\overline{\R}$ and a point $\bar{x}$ with $f(\bar{x})$ finite. Then the {\em subderivative} of $f$ at $\bar{x}$ is defined by 
$$df(\bar{x})(\bar{u}):=\lf_{\substack{t\searrow 0\\ u\to\bar{u}}} \frac{f(\bar{x}+tu)-f(\bar{x})}{t},$$
while for any vector $\bar{v}\in\R^n$, the {\em critical cone of} $f$ {\em at} $\bar{x}$ {\em for} $\bar{v}$ is defined by
$$C_{f}(\bar{x},\bar{v}):=\{u\in \R^n: \langle \bar{v},u\rangle=df(\bar{x})(u)\}.$$
The {\em parabolic subderivative} of $f$ {\em at} $\bar{x}$ {\em for} $\bar{u}\in \dom df(\bar{x})$ {\em with respect to} $\bar{w}$ is  
$$d^2f(\bar{x})(\bar{u}|\bar{w})=\lf_{\substack{t\searrow 0\\ w\to\bar{w}}} \frac{f(\bar{x}+t\bar{v}+\frac{1}{2}t^2w)-f(\bar{x})-df(\bar{x})(\bar{u})}{\frac{1}{2}t^2}.$$}
\end{defn}

Some comments are in order. The directional subderivative $df(\bar{x})(\bar{u})$ simply measures the rate of change of $f$ in direction $\bar{u}$. Whenever $f$ is locally Lipschitz continuous at $\bar{x}$ we may set $u=\bar{u}$ in the definition.
The critical cone $C_{f}(\bar{x},\bar{v})$ denotes the set of directions $u$ along which the directional derivative  at $\bar{x}$ of the function $x\mapsto f(x)-\langle \bar{v},x\rangle$ vanishes. The parabolic subderivative $d^2f(\bar{x})(\bar{u}|\bar{w})$ measures the second order variation of $f$ along points lying on a parabolic arc, and hence the name. In particular, when $f$ is $C^2$ smooth at $\bar{x}$, we have 
$$d^2f(\bar{x})(\bar{u}|\bar{w})=\langle\nabla^2 f(\bar{x})\bar{u},\bar{u}\rangle +\langle \nabla f(\bar{x}),\bar{w}\rangle.$$

This three constructions figure prominently in second-order optimality conditions. Namely, if $\bar{x}$ is a local minimizer of $f$, then $df(\bar{x})(u)\geq 0 \textrm{ for all } u\in \R^n$, and we have $\inf_{w\in \R^n} d^2f(\bar{x})(u|w)\geq 0$ for any nonzero $u\in C_f(\bar{x},0)$. On the other hand, deviating from the classical theory, the assumption $df(\bar{x})(u)\geq 0 \textrm{ for all } u\in \R^n$ along with the positivity  $\inf_{w\in \R^n} d^2f(\bar{x})(u|w)> 0$ for any nonzero $u\in C_f(\bar{x},0)$, guarantees that $\bar{x}$ is a strong local minimizer of $f$ only under additional regularity assumptions on the function $f$. See for example \cite{Bon_Shap} or \cite[Theorem 13.66]{VA} for more details.

We will now see that in the generic semi-algebraic set-up, the situation simplifies drastically: the parabolic subderivative completely characterizes quadratic growth at a critical point. 
The key to the development, not surprisingly, is the relationship between subderivatives of a function $f$ and the subderivatives of the restriction of $f$ to an identifiable manifold.

\begin{thm}[First-order subderivatives and identifiable manifolds]\label{thm:critman}
Consider an lsc function $f\colon\R^n\to\overline{\R}$ and suppose that $f$ admits a $C^2$ identifiable manifold $\mathcal{M}$ at a point $\bar{x}$ for $\bar{v}\in \partial_p f(\bar{x})$. Then for any  $u\in T_{\mathcal{M}}(\bar{x})$ we have
$$df(\bar{x})(u)=d(f+\delta_{\mathcal{M}})(\bar{x})(u)=\langle \bar{v}, u\rangle.$$
\end{thm}
\begin{proof}
Let $g\colon\R^n\to\R$ be a $C^2$-smooth function coinciding with $f$ on $\mathcal{M}$ near $\bar{x}$. Standard subdifferential calculus implies
$$\partial_p f(\bar{x})\subset \partial_p (f+\delta_{\mathcal{M}})(\bar{x})=\partial_p (g+\delta_{\mathcal{M}})(\bar{x})=\nabla g(\bar{x})+N_{\mathcal{M}}(\bar{x}).$$
Moreover, one can easily verify $d(g+\delta_{\mathcal{M}})(\bar{x})(u)=\langle\nabla g(\bar{x}),u\rangle$ for any $u\in T_{\mathcal{M}}(\bar{x})$. Since  by the chain of inclusions above $\bar{v}$ lies in $\nabla g(\bar{x})+N_{\mathcal{M}}(\bar{x})$, we deduce $$d(f+\delta_{\mathcal{M}})(\bar{x})(u)=d(g+\delta_{\mathcal{M}})(\bar{x})(u)=\langle\nabla g(\bar{x}),u \rangle=\langle\bar{v},u\rangle.$$
Now since identifiable manifolds are partly smooth, we have $\para \partial_p f(\bar{x})=N_{\mathcal{M}}(\bar{x})$. Consequently
we deduce
$$\aff \partial_p f(\bar{x}) =\nabla g(\bar{x})+N_{\mathcal{M}}(\bar{x}).$$
In particular, for any $u\in T_{\mathcal{M}}(\bar{x})$ we have equality $\langle \aff \partial_p f(\bar{x}), u\rangle=\langle \nabla g(\bar{x}),u\rangle=\langle \bar{v},u\rangle$. On the other hand $df(\bar{x})$ is the support function of the Fr\'{e}chet subdifferential $\hat{\partial} f(\bar{x})$ (see \cite[Excercise 8.4]{VA}), and since $f$ is prox-regular at $\bar{x}$ for $\bar{v}$, we have $\aff \partial_p f(\bar{x})=\aff \hat{\partial} f(\bar{x})$. 
We conclude $df(\bar{x})(u)=\langle  \bar{v},u\rangle$, as claimed. 
\end{proof}

As a direct consequence, we deduce that critical cones are simply tangent spaces to identifiable manifolds, when the latter exist. The following is an extension of \cite[Proposition 6.4]{dim}.

\begin{thm}[Critical cones and identifiable manifolds]\label{thm:critman2}
Consider an lsc function $f\colon\R^n\to\overline{\R}$ and suppose that $f$ admits a $C^2$ identifiable manifold $\mathcal{M}$ at a point $\bar{x}$ for $\bar{v}\in \partial_p f(\bar{x})$. Then the critical cone coincides with the tangent space
$$C_f(\bar{x},\bar{v})=T_{\mathcal{M}}(\bar{x}).$$
\end{thm}
\begin{proof} 
The inclusion $C_f(\bar{x},\bar{v})\supset T_{\mathcal{M}}(\bar{x})$ is immediate from Theorem~\ref{thm:critman}. Conversely, consider a vector $u\in C_f(\bar{x},\bar{v})$. Since $df(\bar{x})$ is the support function of the Fr\'{e}chet subdifferential $\hat{\partial} f(\bar{x})$ (see \cite[Exercise 8.4]{VA}) and by prox-regularity the subdifferentials $\partial_p f(\bar{x})$ and $\hat{\partial} f(\bar{x})$ coincide near $\bar{v}$, we deduce that $u$ lies in 
$N_{\partial_p f(\bar{x})}(\bar{v})$. On the other hand by Theorem~\ref{thm:basic_prop_ident}, locally near $\bar{v}$, we have equality  $$\partial_p f(\bar{x})=\partial_p (f+\delta_{\mathcal{M}})(\bar{x})=\nabla g(\bar{x})+N_{\mathcal{M}}(\bar{x}),$$
where $g$ is any $C^2$ smooth function agreeing with $f$ on $\mathcal{M}$ near $\bar{x}$. Consequently $u$ lies in $T_{\mathcal{M}}(\bar{x})$, as claimed.
\end{proof}

Next we need set analogues of subderivatives -- first-order and second-order tangent sets. These are obtained by applying the subderivative concepts to the indicator function. More concretely we have the following.

\begin{defn}[First order and second order tangent sets]
{\rm
Consider a set $\Omega\subset\R^n$ and a point $\bar{x}\in \Omega$. Then the {\em tangent cone} to $\Omega$ at $\bar{x}$ is the set
$$T_{\Omega}(\bar{x}):=\{u: \exists t_i\downarrow 0 \textrm{ and } u_i\to u \quad \textrm{ such that }\quad \bar{x}+t_iu_i\in \Omega\},$$
while the {\em critical cone of} $\Omega$ {\em at} $\bar{x}$ {\em for} $\bar{v}$ is defined by
$$C_{\Omega}(\bar{x},\bar{v}):=T_{\Omega}(\bar{x})\cap \bar{v}^{\perp}.$$
The {\em second-order tangent set to} $\Omega$ {\em at} $\bar{x}$ {\em for} $\bar{u}\in T_{\Omega}(\bar{x})$ is the set 
$$T^2_{\Omega}(\bar{x}|\bar{u}):=\{w: \exists t_i\downarrow 0  \textrm{ and } w_i\to w \quad\textrm{ such that } \quad \bar{x}+t_i \bar{u} +\frac{1}{2}t^2_i w_i\in \Omega\}.$$
}
\end{defn}

\smallskip
\noindent One can now easily verify the relationships:
$$T_{\Omega}(\bar{x})=\dom d\delta_Q(\bar{x}),\qquad C_{\Omega}(\bar{x},\bar{v})=C_{\delta_Q}(\bar{x},\bar{v}), \qquad T^2_{\Omega}(\bar{x}|\bar{u})=\dom d^2\delta_{\Omega}(\bar{x})(\bar{u}|\cdot).$$
Next we record an important relationship between projections and identifiable manifolds \cite[Proposition 4.5]{shanshan}. Naturally, we say that a set $\mathcal{M}$ is a $C^p$ {\em identifiable manifold relative to a set} $Q$ {\em at} $\bar{x}$ {\em for} $\bar{v}\in N_Q(\bar{x})$ whenever $\mathcal{M}$ is a $C^p$ identifiable manifold relative to the indicator function $\delta_Q$ at $\bar{x}$ for $\bar{v}\in \partial\delta_Q(\bar{x})$.
\begin{prop}[Projections and identifiability]\label{prop:prox_ident}
Consider a closed set $Q\subset\R^n$ and suppose that $\mathcal{M}$ is a $C^p$-identifiable manifold ($p\geq 2$) at $\bar{x}$ for $\bar{v}\in N^p_Q(\bar{x})$. Then for all sufficiently small $\lambda >0$, the projections $P_Q$ and $P_{\mathcal{M}}$ coincide on a neighborhood of $\bar{x}+\lambda\bar{v}$ and are $C^{p-1}$-smooth there.
\end{prop}

\begin{prop}[Second-order tangents to sets with identifiable structure]\label{prop:tan_id}
Suppose that a closed set $Q\subset \R^n$ admits an identifiable $C^3$ manifold at $\bar{x}$ for $\bar{v}\in N^p_Q(\bar{x})$. Consider a nonzero tangent $\bar{u}\in T_{\mathcal{M}}(\bar{x})$ and a vector $\bar{w}\in T^2_{Q}(\bar{x}|\bar{u})$. Then for any real $\epsilon >0$, there exists $\hat{u}\in T_{\mathcal{M}}(\bar{x})$ and $\hat{w}\in T^2_{\mathcal{M}}(\bar{x}|\hat{u})$ satisfying 
$$|\bar{u}-\hat{u}|\leq \epsilon \qquad \textrm{ and }\qquad \langle \bar{v},\hat{w}\rangle\geq \langle \bar{v},\bar{w}\rangle.$$
\end{prop}
\begin{proof}
By definition of $\bar{w}$, there exist numbers $t_i\downarrow 0$ and vectors $w_i\to \bar{w}$ so that the points $x_i:=\bar{x}+t_i\bar{u}+\frac{1}{2}t^2_i w_i$ lie in $\Omega$ for each $i$.
By Proposition~\ref{prop:prox_ident}, we may
choose $r >0$ satisfying $P_{Q}(\bar{x}+r \bar{v})= \bar{x}$, so that $P_Q$ coincides with $P_{\mathcal{M}}$ on a neighborhood of $\bar{x} +r \bar{v}$, and so that $P_Q$ is $C^2$-smooth on this neighborhood. Define now
$z_i=P_Q(x_i+r\bar{v})$. 
Since $P_Q$ is $C^2$-smooth on a neighborhood of $\bar{x}+rv$, we may write
$z_i=\bar{x}+t_i \hat{u} +\frac{1}{2}t_i^2\hat{w}_i$ for some $\hat{u}\in T_{\mathcal{M}}(\bar{x})$ and some $\hat{w}_i$ converging to a vector $\hat{w}\in T^2_{\mathcal{M}}(\bar{x}|\hat{u})$. It is standard that the derivative $\nabla P_{\mathcal{M}}(\bar{x})$ coincides with the linear projection onto the tangent space $T_{\mathcal{M}}(\bar{x})$, and hence decreasing $r$ we may ensure $|u-\hat{u}|<\epsilon$.
By definition of $z_i$ then we have the inequality
$$|x_i-z_i+r\bar{v}|\leq r|\bar{v}|,$$
and hence
$$\langle \bar{v},z_i-x_i\rangle\geq \frac{1}{2r}|x_i-z_i|^2\geq 0.$$
We deduce 
$$0\leq \langle \bar{v}, t_i(\hat{u}-\bar{u})+\frac{1}{2}t^2_i(\hat{w}_i-w_i)\rangle=\frac{1}{2}t^2_i\langle \bar{v},\hat{w}_i-\bar{w}_i\rangle.$$
Dividing by $\frac{1}{2}t_i^2$ and taking the limit the result follows.
\end{proof}

Finally, we arrive at the key relationship between the parabolic subderivative of a function and that of its restriction to an identifiable manifold.
\begin{cor}[Second-order subderivatives and identifiability]\label{cor:sec_order}
Suppose that an lsc function $f\colon\R^n\to\overline{\R}$ admits an identifiable $C^3$ manifold $\mathcal{M}$ at $\bar{x}$ for $0\in \partial_p f(\bar{x})$. Consider a nonzero vector $\bar{u}\in T_{\mathcal{M}}(\bar{x})$ and a vector $\bar{w}$. Then for any real $\epsilon >0$, there exists $\hat{u}\in T_{\mathcal{M}}(\bar{x})$ and $\hat{w}$ satisfying $|\bar{u}-\hat{u}|\leq \epsilon$ and 
$$d^2 f(\bar{x})(\bar{u}|\bar{w}) \geq d^2 (f+\delta_{\mathcal{M}})(\bar{x})(\hat{u}|\hat{w}).$$
\end{cor}
\begin{proof}
By \cite[Proposition 3.14]{ident}, the set $\mathcal{K}:=\gph (f+\delta_{\mathcal{M}})$ is a $C^3$ identifiable manifold relative to $\epi f$ at $(\bar{x},f(\bar{x}))$ for $(\bar{v},-1)$. Moreover by Theorem~\ref{thm:critman}, we have 
$$T_{\mathcal{K}}(\bar{x})=\{(u,\alpha): u\in T_{\mathcal{M}}(\bar{x}) \quad \textrm{ and }\quad \alpha=df(\bar{x})(u)\}.$$
Define $\bar{\beta} :=df(\bar{x})(\bar{u})$. Then
by \cite[Example 13.62]{VA}, equality
$$\epi d^2 f(\bar{x})(\bar{u}|\cdot)=T^2_{\sepi f}((\bar{x},f(\bar{x}))|(\bar{u},\bar{\beta})),$$
holds. Define $\bar{r}:=d^2 f(\bar{x})(\bar{u}|\bar{w})$.
Applying Proposition~\ref{prop:tan_id}, we deduce that
there exist $(\hat{u},\hat{\beta})\in T_{\mathcal{K}}(\bar{x},f(\bar{x}))$ and $(\hat{w},\hat{r})\in T^2_{\mathcal{K} }((\bar{x},f(\bar{x}))|(\hat{u},\hat{\beta}))$ satisfying 
$$|(\bar{u},\bar{\beta})-(\hat{u},\hat{\beta})|\leq \epsilon\quad \textrm{ and } \quad \langle (0,-1), (\hat{w},\hat{r}) \rangle \geq \langle (0,-1), (\bar{w},\bar{r})\rangle.$$
Clearly $\hat{\beta}=df(\bar{x})(\hat{u})$ and $\hat{r}=d^2 (f+\delta_{\mathcal{M}})(\bar{x})(\hat{u}|\hat{w})$.
We deduce
$$d^2 f(\bar{x})(\bar{u}|\bar{w}) \geq d^2 (f+\delta_{\mathcal{M}})(\bar{x})(\hat{u}|\hat{w}),$$
as claimed.
\end{proof}

We now arrive at the main result of this section.
\begin{thm}[Generic properties of semi-algebraic problems]\label{thm:final_uncon}
Consider an lsc, semi-algebraic function $f\colon\R^n\to\overline{\R}$. Then there exists an integer $N >0$ such that for a generic vector $v\in\R^n$ the function 
$$f_v(x):=f(x)-\langle v,x\rangle$$
has no more than $N$ critical points. Moreover each such critical point $\bar{x}$ satisfies 
\begin{enumerate}
\item {\bf (prox-regularity)} $f_v$ is prox-regular at $\bar{x}$ for $0$.
\item {\bf (strict complementarity)} The inclusion $0\in \ri\, \partial_p f_v(\bar{x})$ holds.
\item {\bf (identifiable manifold)} $f_v$ has an identifiable manifold $\mathcal{M}$ at $\bar{x}$ for $0$.
\item {\bf (smooth dependence of critical points)} The subdifferential $\partial f$ is strongly regular at $(\bar{x},v)$. More precisely, 
there exist neighborhoods $U$ of $\bar{x}$ and $V$ of $v$ so that the critical point mapping
$$w\mapsto U\cap (\partial f)^{-1}(w)=\{x\in U: x \textrm{ is critical for } f(\cdot)-\langle w,\cdot\rangle\}$$
is single-valued and analytic on $V$, and maps $V$ onto $\mathcal{M}$.
\end{enumerate}
Moreover the following are all equivalent
\begin{enumerate}[$(i)$]
\item\label{it_11} $\bar{x}$ is a local minimizer of $f_v$.
\item\label{it_2} $\bar{x}$ is a stable strong local minimizer of $f_v$.
\item\label{it_3} The inequality $$\inf_{w\in \R^n} d^2 f_v(\bar{x})(u|w)>0\qquad \textrm{ holds for all }\quad 0\neq u\in C_f(\bar{x},v).$$
\item\label{it_4} The inequality $$\inf_{w\in \R^n} d^2 (f_v+\delta_{\mathcal{M}})(\bar{x})(u|w)>0\qquad \textrm{ holds for all }\quad 0\neq u\in T_{\mathcal{M}}(\bar{x}).$$
\end{enumerate}
\end{thm}
\begin{proof}
In light of Corollary~\ref{cor:basic2}, we must only argue the claimed equivalence of the four properties. To this end, observe that for generic $v$, the equivalence $(\ref{it_11})\Leftrightarrow(\ref{it_2})$ was established in Corollary~\ref{cor:basic2}. On the other hand, Theorem~\ref{thm:basic_prop_ident} shows that $(\ref{it_2})$ is equivalent to $\bar{x}$ being a strong local minimizer of $f_v$ on $\mathcal{M}$, which in turn for classical reasons is equivalent to $(\ref{it_4})$. Note also that the implication $(\ref{it_3})\Rightarrow(\ref{it_4})$ is obvious from Theorem~\ref{thm:critman2}.
Thus we must only show the implication $(\ref{it_4})\Rightarrow(\ref{it_3})$, but this follows immediately from Corollary~\ref{cor:sec_order}.
\end{proof}

Note that property $(\ref{it_4})$ in the theorem above involves only classical analysis.

\section{Composite semi-algebraic optimization}\label{sec:Lag} 
In this section, we consider composite optimization problems of the form
$$\min f(x)+h(G(x)),$$
where $f\colon\R^n\to\overline{\R}$ and $h\colon\R^m\to\overline{\R}$ are lsc functions and $G\colon\R^n\to\R$ is $C^2$-smooth.
A prime example is the case of smoothly constrained optimization; this is the case where $h$ is the indicator function of a closed set.
We call a point $x\in\R^n$ {\em composite critical} for the problem if there exists a vector  
\begin{equation*}
\lambda\in \partial h(G(x))\quad \textrm{ satisfying }\quad
 -\nabla G(x)^*\lambda\in \partial f(x).
\end{equation*}
Whenever the optimality condition above holds, we call $\lambda$ a {\em Lagrange multiplier vector} and the tuple $(x,\lambda)$ a {\em composite critical pair}. The multiplier $\lambda$ is sure to be unique under the condition:
\begin{equation}\label{eqn:strong}
\para \partial h(G(x))\bigcap [\nabla G(x)^*]^{-1}\para \partial f(x)=\{0\}.
\end{equation}
Indeed, this is a direct analogue of the linear independence constraint qualification in nonlinear programming.

In general, the notion of composite criticality is different from criticality (as defined in the previous sections) for the function $f+h\circ G$. If $x$ is a critical point of $f+h\circ G$, then $x$ is composite critical only under some additional condition, such as the {\em basic constraint qualification}
\begin{equation}\label{eqn:bcq}
\partial^{\infty} h(G(x))\bigcap [\nabla G(x)^*]^{-1}\partial^{\infty} f(x)=\{0\}.
\end{equation}
This qualification condition is a generalization of the Mangasarian-Fromovitz constraint qualification in nonlinear programming and is in particular implied by \eqref{eqn:strong}; see the discussion in \cite{lagr} for more details. 
Conversely, if $x$ is a composite critical point and both $f$ and $h$ are subdifferentially regular \cite[Definition 7.25]{VA} (as is the case when $f$ and $h$ are convex), then $x$ is also a critical point of the function $f+h\circ G$.

In this section, we consider properties of composite critical points for generic composite semi-algebraic problems. To this end, we will assume that $f$, $G$, and $h$ are all semi-algebraic and we will consider the canonically perturbed problems:
$$\min f(x)+h(G(x)+y)-\langle v,x\rangle.$$
Then composite criticality is succinctly captured by the generalized equation
\begin{equation}\label{eqn:var_ineq}
\begin{bmatrix}
v \\
y
\end{bmatrix}\in\begin{bmatrix}
\nabla G(x)^*\lambda \\
-G(x)
\end{bmatrix}+ \Big(\partial f\times (\partial h)^{-1}\Big)(x,\lambda).
\end{equation}
The path to generic properties is now clear since the perturbation parameters $(v,y)$ appear in the range space of a semi-algebraic set-valued mapping having a small graph.

Before we proceed, we briefly recall that subderivatives admit a convenient calculus \cite[Exercise 13.63]{VA} for the composite problem.
In what follows, for any $C^2$-smooth mapping $G(x)=(g_1(x),\ldots, g_m(x))$ we use the notation
$$\nabla^2 G(x)[u,u]=\big(\langle\nabla^2 g_1(x)u,u\rangle,\ldots, \langle\nabla^2 g_m(x)u,u\rangle\big).$$

\begin{thm}[Calculus of subderivatives]\label{thm:calc} {\hfill \\ }
Consider a $C^2$-smooth mapping $G\colon\R^n\to\R^m$ and lsc functions $f\colon\R^n\to\overline{\R}$ and $h\colon\R^m\to\overline{\R}$. Suppose that a point $x$ satisfies the constraint qualification 
$$\partial^{\infty} h(G(x))\bigcap [\nabla G(x)^*]^{-1}\partial^{\infty} f(x)=\{0\}.$$
Then the equality 
$$d (f+h\circ G)(x)(u)=df(x)(u)+dh(G(x))(\nabla G(x)u))$$
holds.
Moreover for any $u$ with $d(f+h\circ G)(x)(u)$ finite, we have
$$d^2 (f+h\circ G)(x)(u|w)=d^2f(x)(u|w)+d^2h(G(x))\Big(\nabla G(x)u\,\Big|\, \nabla^2 G[u,u]+\nabla G(x)w\Big).$$
\end{thm}

We are now ready to prove the main result of this section.
Note that if for almost every $v$, a property is valid for almost every $y$ (with the $v$ fixed), then by Fubini's theorem the said property holds for almost every pair $(v,y)$. The same holds with $v$ and $y$ reversed. We will use this observation implicitly throughout.
\begin{thm}[Generic properties of composite optimization problems]\label{thm:main} Consider a $C^2$-smooth semi-algebraic mapping
$G\colon\R^n\to\R^m$ and lsc semi-algebraic functions $f\colon\R^n\to\overline{\R}$ and $h\colon\R^m\to\overline{\R}$. Define now the family of composite optimization problems $P(v,y)$  by
$$\min\, f_v(x)+h(G_y(x)),$$
under the perturbations $f_v(x):=f(x)-\langle v,x\rangle$ and $G_y(x)=G(x)+y$.
Then for almost every $y\in \R^m$, the qualification conditions 
\begin{align}
\spann \partial^{\infty} h(G_y(x))\bigcap [\nabla G(x)^*]^{-1}\spann \partial^{\infty} f_v(x)&=\{0\},\label{eqn:qual_cond1} \\
\para \partial h(G_y(x))\bigcap [\nabla G(x)^*]^{-1}\para \partial f_v(x)&\subseteq\{0\},\label{eqn:qual_cond2}
\end{align}
hold for any $x$, for which $f_v(x)$ and $h(G_y(x))$ are finite.
 Moreover there exists an integer $N >0$ such that for a generic collection of parameters $(v,y)\in\R^n\times\R^m$, the problem  $P(v,y)$ has at most $N$ composite critical points, and for any such composite critical point $\bar{x}$ of $P(v,y)$, there exist a unique Lagrange multiplier vector
 \begin{equation*}
 \bar{\lambda}\in \partial h(G_y(\bar{x}))\quad \textrm{ satisfying }\quad
  -\nabla G(\bar{x})^*\bar{\lambda}\in  \partial f_v(\bar{x}).
 \end{equation*}
Moreover, defining $\bar{w}:=-\nabla G(\bar{x})^*\bar{\lambda}$, the following are true. 
\begin{enumerate}
\item {\bf (prox-regularity)} $f_v$ is prox-regular at $\bar{x}$ for $\bar{w}$ and $h$ is prox-regular at $G_y(\bar{x})$ for $\bar{\lambda}$.
\item\label{it:strict_comp} {\bf (strict-complementarity)} The inclusions 
$$\bar{\lambda}\in \ri \partial_p h(G_y(\bar{x}))\quad \textrm{ and }\qquad \bar{w}\in \ri \partial_p f_v(\bar{x})\qquad\textrm{ hold}.$$ 
\item {\bf (identifiable manifold)} $f_v$ admits a $C^{\omega}$ identifiable manifold $\mathcal{M}$ at $\bar{x}$ for $\bar{w}$ and $h$ admits a $C^{\omega}$ identifiable manifold $\mathcal{K}$ at $G_{y}(\bar{x})$ for $\bar{\lambda}$. 
\item {\bf (nondegeneracy)} The constraint qualification (nondegeneracy condition)
\begin{align*}
N_{\mathcal{K}}(G_{y}(\bar{x}))\cap [\nabla G(\bar{x})^*]^{-1}N_{\mathcal{M}}(\bar{x})&=\{0\} \qquad \textrm{ holds}. \label{eqn:part_basic} 
\end{align*}
\item\label{it:smooth_dep} {\bf  (smooth dependence of critical triples)}  The mapping 
$$(\hat{v},\hat{y})\mapsto \Big\{(x,\lambda): \textrm{ the pair } (x,\lambda) \textrm{ is composite critical for }P(\hat{v},\hat{y})\Big\}$$ 
admits a single-valued analytic localization around $(v,y,\bar{x},\bar{\lambda})$.
\end{enumerate}
Moreover the following are equivalent.
\end{thm}
\begin{enumerate}[$(i)$]
\item $\bar{x}$ is a local minimizer of $P(v,y)$.
\item $\bar{x}$ is a strong local minimizer of $P(v,y)$.
\item The inequality $$d^2 f(\bar{x})(u|z)+d^2 h(G_y(\bar{x}))\Big(\nabla G(\bar{x})u\,\Big|\, \nabla^2 G(\bar{x})[u,u]+\nabla G(\bar{x})z\Big) >0$$ 
holds for all  nonzero $u\in C_f(\bar{x},\bar{w})\cap[\nabla G(\bar{x})]^{-1}C_h(G_{y}(\bar{x}),\bar{\lambda})$ and all $z\in\R^n$.
\item The inequality
$$d^2 (f+\delta_{\mathcal{M}})(\bar{x})(u|z)+d^2 (h+\delta_{\mathcal{K}})\big(G_y(\bar{x})\big)\Big(\nabla G(\bar{x})u\,\Big|\, \nabla^2 G(\bar{x})[u,u]+\nabla G(\bar{x})z\Big) >0$$ 
holds for all  nonzero $u\in T_{\mathcal{M}}(\bar{x})\cap [\nabla G(\bar{x})]^{-1}T_{\mathcal{K}}(G_{y}(\bar{x}))$ and all $z\in\R^n$.
\end{enumerate}
\begin{proof}
First applying \cite[Lemma 8]{Lewis-Clarke} and Theorem~\ref{thm:strat_exist}, we obtain a $C^{\omega}$ stratification $\{A_i\}$ of $\dom f$ and a $C^{\omega}$ stratification $\{B_j\}$ of $\dom h$ having the property that $f$ is $C^{\omega}$-smooth on each $A_i$ and $h$ is $C^{\omega}$-smooth on each $B_j$, and so that 
$$\partial^{\infty}f(x)\cup\para \partial f(x)\subset N_{A_i}(x) \quad \textrm{ and } \quad \partial^{\infty}h(z) \cup\para \partial h(z)\subset N_{B_j}(z)$$
for any $x\in A_i$ and $z\in B_j$.
For fixed indices $i$ and $j$, the standard Sard's theorem implies that for almost every $y\in\R^m$, the restriction of $G_{y}$ to $A_i$ is transverse to $B_j$, that is for any $x\in A_i$ with $G_{y}(x)\in B_j$ we have 
$$N_{B_j}(G_{y}(x))\cap [\nabla G(x)^*]^{-1}N_{A_i}(x)=\{0\}.$$
Since there are finitely many indices $i$ and $j$, the claimed qualification conditions \eqref{eqn:qual_cond1} and \eqref{eqn:qual_cond2} follow.

Define now the set-valued mapping $\mathcal{I}\colon\R^n\times\R^m\rightrightarrows\R^n\times\R^m$ by
$$\mathcal{I}(x,\lambda)=\begin{bmatrix}
\nabla G(x)^*\lambda \\
-G(x)
\end{bmatrix}+ \Big(\partial f\times (\partial h)^{-1}\Big)(x,\lambda).
$$
Observe $(v,y)\in \mathcal{I}(x,\lambda)$ if and only if $(x,\lambda)$ is a composite critical pair for $P(v,y)$.
It is easy to see, in turn, that $\gph \mathcal{I}$ is $C^1$ diffeomorphic to $\gph \partial f\times \gph(\partial h)^{-1}$, and hence by \cite[Theorem~3.7]{dim} has dimension $n+m$. 
Applying the semi-algebraic Sard's theorem for weakly critical values (Theorem~\ref{thm:sard}), we deduce that there exists an integer $N >0$ such that for generic parameters $(v,y)$, the problem $P(v,y)$ has at most $N$ composite critical points $x$. 
Moreover for any composite critical point $\bar{x}$ of $P(v,y)$, the Lagrange multiplier vector $\bar{\lambda}$ is unique  
 for almost every $(v,y)$ by inclusion \eqref{eqn:qual_cond2}.

We now prove the strict complementarity claim. To this end, define the mapping
$$\mathcal{I}_p(x,\lambda)=\begin{bmatrix}
\nabla G(x)^*\lambda \\
-G(x)
\end{bmatrix}+ \Big(\ri \partial_p f\times \big(\ri \partial_p h\big)^{-1}\Big)(x,\lambda).
$$
Clearly the inclusion $\gph \mathcal{I}_p\subset \gph \mathcal{I}$ holds, and by what we have already proved both mappings $\mathcal{I}_p$ and $\mathcal{I}$ are finite valued almost everywhere. We now claim that $\gph \mathcal{I}_p$ is dense in $\gph \mathcal{I}$. To see this, fix a pair $(v,y)\in \mathcal{I}(x,\lambda)$. Equivalently we may write $$0= w + \nabla G(x)^*\lambda, \quad \textrm{ for some } w\in\partial f_v(x) \textrm{ and } \lambda\in \partial h(G_y(x)).$$ By definition of the limiting subdifferential, there are sequences $(x_k,u_k)\to (x,w+v)$ in $\gph (\ri\partial_p f)$ and $(z_k,\lambda_k)\to (G_y(x),\lambda)$ in $\gph (\ri\partial_p h)$. 
Defining $\gamma_k:=(u_k-(w+v))+(\nabla G(x_k)^*\lambda_k-\nabla G(x)^*\lambda)$ and $\alpha_k:=z_k-G_y(x_k)$ it is easy to verify the inclusion
$$(v+\gamma_k, y+\alpha_k)\in \mathcal{I}_p(x_k,\lambda_k).$$ Hence $\gph \mathcal{I}_p$ is dense in $\gph \mathcal{I}$. Since both $\mathcal{I}^{-1}$ and $\mathcal{I}^{-1}_p$ are semi-algebraic and finite almost everywhere, it follows immediately that $\mathcal{I}^{-1}$ and $\mathcal{I}^{-1}_p$ agree almost everywhere on $\R^n \times \R^m$. This establishes the strict complementarity claim \ref{it:strict_comp}.

Moving on to existence of identifiable manifolds, 
applying Theorem~\ref{thm:sard} to the mapping $\mathcal{I}_p$, we deduce that there exists an integer $N$, a finite collection of open semi-algebraic sets $\{U_i\}^N_{i=0}$ in $\R^n\times \R^m$, and analytic semi-algebraic single-valued mappings  $$E_i^j\colon U_i\to\R^n\times\R^m \qquad\textrm{ for } i=0,\ldots,N \textrm{ and } j=1,\ldots, i$$ satisfying:
\begin{enumerate}
\item $\bigcup_i U_i$ is dense in $\R^n\times \R^m$.
\item For any $(v,y)\in U_i$, the image $\mathcal{I}^{-1}_p(v,y)$ has cardinality $i$.
\item We have the representation
$$\mathcal{I}^{-1}_p(v,y)  =  \{E_i^j(v,y) : j=1,2,\ldots,i\} \qquad\textrm{ whenever }(v,y)\in U_i.$$
\end{enumerate}

Let  $X^j_i(v,y)$ denote the composition of $E^j_i$ with the projection $(x,\lambda)\mapsto x$ and let $F^j_i(v,y):=G(X^j_i(v,y))+y$. 
Applying Theorem~\ref{thm:strat_exist} to each $X_i^j$ and $F^j_i$, we may find a dense open subset $\widehat{U}_i$ of $U_i$ so that 
\begin{itemize}
\item $f$ is analytic on $X^j_i(\widehat{U}_i)$ and $h$ is analytic on $F^j_i(\widehat{U}_i)$.
\item $X_i^j$ and $F^j_{i}$ are analytic and have constant rank on $\widehat{U}_i$
\end{itemize}

Let $(x,\lambda)$ be such that $X_i^j(v,y)=x$ and so that $(x,\lambda)$ is a composite critical pair for $P(v,y)$. Define also $w:=-\nabla G(x)^*\lambda$. Then due to the constant rank, there exists a neighborhood $W$ of $(v,y)$ so that $X_i^j(W)$ and $F_i^j(W)$ are analytic manifolds. 
We claim that $X_i^j(W)$ is an identifiable manifold relative to $f_v$ at $x$ for $w$ and that $F_i^j(W)$ is an identifiable manifold relative to $h$ at $G_y(x)$ for $\lambda$. 

To see this, consider sequences $(x_k,w_k)\to (x,w)$ in $\gph \partial f_v$ and $(z_k,\lambda_k)\to (G_y(x),\lambda)$ in $\gph \partial h$. 
Defining $\gamma_k:=(w_k-w)+(\nabla G(x_k)^*\lambda_i-\nabla G(x)^*\lambda)$ and $\alpha_k:=z_k-G_y(x_k)$ we have the inclusion
$$(v+\gamma_k, y+\alpha_k)\in \mathcal{I}_p(x_k,\lambda_k).$$ 
Hence for all large indices $k$ equality
$$E^{i}_j(v+\gamma_k, y+\alpha_k)=(x_k,\lambda_k)$$
holds.
We deduce for sufficiently large $k$ the inclusion $x_k\in X_i^j(W)$. Hence $X_i^j(W)$ is indeed identifiable relative to $f_v$ at $x$ for $w$. Moreover, we have $z_k=F^j_{i}(v+\gamma_k, y+\alpha_k)\in F^j_i(W)$ for all large $k$. We conclude that $F_i^j(W)$ is  identifiable relative to $h$ at $G_y(x)$ for $\lambda$, as claimed.
The nondegeneracy claim is a simple consequence of the construction and the classical Sard's theorem. Finally the four equivalent properties are immediate from 
Theorems~\ref{thm:final_uncon} and \ref{thm:calc}.

\end{proof}

Note that Theorem~\ref{thm:main}  with $h=0$ and $G=I$ reduces to Theorem~\ref{thm:final_uncon}.
It is interesting to reinterpret Theorem~\ref{thm:main} in the convex setting. To this end, recall that for any convex function $f\colon\R^n\to\overline{\R}$, the {\em Fenchel conjugate} $f^*\colon\R^n\to\overline{\R}$ is defined by
$$f^*(u):=\sup_x \,\{\langle u,x \rangle -f(x)\},$$
and the relationship $\partial f^*=(\partial f)^{-1}$ holds.

Fix now a linear mapping
$A\colon\R^n\to\R^m$ and lsc convex functions $f\colon\R^n\to\overline{\R}$ and $h\colon\R^m\to\overline{\R}$. Within the Fenchel framework, we consider the family of primal optimization problems given by
$$\inf_x\, f(x)+h(Ax+y)-\langle v,x\rangle,$$
and associate with them the dual problems
$$\sup_u\, -h^*(u)-f^*(v-A^*u)+\langle y,u\rangle.$$
Then the primal problem is feasible whenever $y$ lies in the set 
$$Y:=\dom h- A(\dom f),$$
and the dual is feasible whenever $v$ lies in 
$$V:=\dom f^*+ A^*(\dom h^*).$$
Standard Fenchel duality then asserts that for $y$ in the interior of $Y$, the primal and dual optimal values are equal and the dual is attained when finite. 
Assuming in addition that $v$ lies in the interior of $V$, optimality is characterized by the generalized equation
$$\begin{bmatrix}
v \\
y
\end{bmatrix}\in\begin{bmatrix}
A^*u \\
-Ax
\end{bmatrix}+ \Big(\partial f\times \partial h^*\Big)(x,u).
$$
This is precisely an instance of the variational inequality \eqref{eqn:var_ineq} in a convex setting. Assuming now that $f$ and $h$ are semi-algebraic, and applying Theorem~\ref{thm:main}, we deduce that for generic parameters $(v,y)$, if the primal and dual problems are feasible then the interiority conditions hold, and both the primal and the dual admit at most one minimizer. Moreover for any such minimizers $x$ and $u$, strict complementarity holds for the primal and the dual, identifiable manifolds exist for both problems, both objectives grow quadratically around $x$ and $u$, respectively,  and the minimizers $x$ and $u$ jointly vary analytically with the parameters $(v,y)$.

\bibliography{dim_graph}{}
\bibliographystyle{plain}

\end{document}